\documentclass[10pt,twoside]{article}
\usepackage[left=25.4mm, right=25.4mm, top=25.4mm, bottom=25.4mm]{geometry}   
\usepackage[utf8]{inputenc}                              
\usepackage[english]{babel}                              
\usepackage[hidelinks]{hyperref}                         
\usepackage{booktabs}                                    
\usepackage{amsmath, amsfonts}                           
\usepackage{caption, subcaption}                         
\usepackage{tikz-cd}                                     
\usepackage{enumitem}                                    
\usepackage{xcolor, graphicx}                            
\usepackage[noblocks]{authblk}                           
\usepackage{titlesec}                                    
\usepackage[T1]{fontenc}                                 
\usepackage{fourier}                                     
\usepackage[cal=cm]{mathalpha}                           
\usepackage{fancyhdr}                                    
\usepackage{etoolbox, cite}
\apptocmd{\thebibliography}{\setlength{\itemsep}{-3pt}}{}{}
\usepackage{amsthm}
\theoremstyle{plain}
\newtheorem{theorem}{Theorem}[section]

\newtheorem{proposition}[theorem]{Proposition}
\newtheorem{corollary}[theorem]{Corollary}

\theoremstyle{definition}
\newtheorem{definition}[theorem]{Definition}
\newtheorem{example}[theorem]{Example}

\theoremstyle{remark}

\newcommand{\R}{\mathbb R}

\newcommand{\G}{\mathcal G}
\newcommand{\J}{\mathcal J}
\newcommand{\F}{\mathcal F}
\newcommand{\K}{\mathcal K}

\newcommand{\xm}{\mathfrak X(M)}
\newcommand{\fm}{\mathfrak F(M)}
\newcommand{\dm}{\Lambda^1 (M)}
\newcommand{\GE}{\Gamma (E)}
\newcommand{\GTM}{\Gamma (\mathbb TM)}
\newcommand{\TM}{\mathbb TM}
\newcommand{\TTM}{TM\oplus T^*M}

\newcommand{\al}{\alpha}
\newcommand{\var}{\varepsilon}

\makeatletter
\def\blfootnote{\gdef\@thefnmark{}\@footnotetext}
\makeatother

\title{\bf Metric polynomial structures on generalized geometry
  \blfootnote{\em E-mail addresses:}
  \blfootnote{Fernando Etayo: \href{mailto:fernando.etayo@unican.es}{fernando.etayo@unican.es}}
  \blfootnote{Pablo Gómez-Nicolás: \href{mailto:pablo.gomeznicolas@unican.es}{pablo.gomeznicolas@unican.es}}
  \blfootnote{Rafael Santamaría: \href{mailto:rsans@unileon.es}{rsans@unileon.es}}
}
\author[*]{Fernando Etayo}
\author[*]{Pablo Gómez-Nicolás}
\author[$\dag$]{Rafael Santamaría}
\affil[*]{\footnotesize Departamento de Matemáticas, Estadística y Computación, Facultad de Ciencias, Universidad de Cantabria, Avda. de los Castros, s/n, 39071, Santander, Spain}
\affil[$\dag$]{\footnotesize Departamento de Matemáticas, Escuela de Ingenierías Industrial, Informática y Aeroespacial, Universidad de León, Campus de Vegazana, 24071, León, Spain}
\date{\small \today}

\begin{document}

\maketitle

\begin{abstract}
\noindent In this document, we study the interaction between different geometric structures that can be defined as morphisms of sections of the generalized tangent bundle $\TM := \TTM\to M$. In particular, we show the behaviour of various generalized polynomial structures with respect to different metrics on $\TM$, in the frame of $(\al, \var)$-metric structures; and the commutation or anti-commutation of generalized polynomial structures, forming triple structures.
\end{abstract}

{\noindent\small {\bf 2020 Mathematics Subject Classification:} 53D18, 53C15, 53D05}

{\noindent\small {\bf Keywords:} Generalized tangent bundle, almost product Riemannian structure, almost para-Hermitian structure, almost Hermitian structure, almost Norden structure, triple structure, generalized Kähler geometry}

\thispagestyle{empty}
\section{Introduction}
\label{SECTION:INTRODUCTION}

The generalized tangent bundle was introduced by N. Hitchin in 2003 in \cite{HITCHIN2003}, and it was later studied in depth by M. Gualtieri \cite{GUALTIERI2004, GUALTIERI2011}. This vector bundle is defined as the Whitney sum of the tangent and the cotangent bundle of the base manifold $M$, that is, $\TM := \TTM\to M$. The first studies in this area were focused in generalized complex structures (by M. Gualtieri, \cite{GUALTIERI2011}) and in generalized paracomplex structures (by A. Wade, \cite{WADE2004}). These morphisms can be defined as metric polynomial structures on $\TM$ which are also requiered to be compatible with a metric on the bundle $\TM$ that emerges naturally. Both generalized complex and paracomplex structures have been analysed since then (e.g. \cite{CULMAGODOYSALVAI2020}), focusing in specific examples and using different points of view. For example, there are multiple generalized complex and paracomplex structures that can be induced using geometric structures defined on a manifold $M$.

Although originally generalized complex and paracomplex structures were asked to be compatible with the natural metric on $\TM$ mentioned before, later works as \cite{NANNICINI2006, NANNICINI2010, NANNICINI2013, NANNICINI2016, NANNICINI2019} by A. Nannicini suggest that we may omit this condition and study a wider range of structures. If we do this, we can define these generalized structures in a similar way than complex and paracomplex structures defined on a manifold. This terminology agrees with specialised texts in vector bundles, such as \cite{POOR1981} by W. A. Poor, where geometric structures on a vector bundle $E\to M$ are defined as morphisms of sections of $E$. Following this line, we define geometric structures for any vector bundle $E$ and specialise them to the case $E = \TM$. We also use a different terminology when a generalized complex or paracomplex structure is compatible with a metric defined on $\TM$, in a similar way than metric polynomial structures defined on the manifold \cite{BURESVANZURA1976}.

In this document we study the interaction of different polynomial structures defined on $\TM$ with other geometric structures, in particular with generalized metrics. The paper is structured as follows:

In Section \ref{SECTION:PRELIMINARIES}, based in classical texts as \cite{POOR1981}, we define different geometric structures on any vector bundle $E\to M$ such as metrics, symplectic structures and polynomial structures. We also check the possible interactions between a polynomial structure and a metric in the form of $(\al, \var)$-metric structures; and the interactions between two polynomial structures in the form of triple structures. After that, we specialise these definitions for the generalized tangent bundle $E = \TM$, also called big tangent bundle in some references (e.g. \cite{VAISMAN2015}). We fix the notation and show the structures that arise naturally in the bundle: the pseudo-Riemannian metric $\G_0$, the symplectic structure $\Omega_0$ (both presented in \cite{HITCHIN2003, GUALTIERI2004}) and the paracomplex structure $\F_0$ (introduced in \cite{WADE2004}).

Section \ref{SECTION:GENERALIZEDMETRICS} focuses in the metrics that can be defined on $\TM$. We show that any generalized metric or generalized symplectic structure is related with the canonical metric $\G_0$. We also discuss the most straightforward way to obtain a generalized metric using a metric defined on the base manifold.

In Section \ref{SECTION:GENERALIZEDALVARMETRICSTRUCTURES} we check the interaction between various polynomial structures and the canonical metric or the generalized induced metric by a metric on the manifold. These interactions are studied in the frame of generalized $(\al, \var)$-metric structures. We distinguish between two types of generalized $(\al, \var)$-metric structures: the natural ones, that involve the natural generalized metric $\G_0$ or the natural generalized paracomplex structure $\F_0$; and the induced ones, that implicate induced generalized structures on $\TM$.

Finally, in Section \ref{SECTION:GENERALIZEDTRIPLESTRUCTURES} we verify which generalized almost complex and almost product structures satisfy properties of commutation or anti-commutation, conforming triple structures on the generalized tangent bundle. We also focus in a subset of the generalized almost bicomplex structures which has been heavily studied in the literature. These structures, called generalized Kähler structures, were introduced in \cite{GUALTIERI2004} and have been studied by many other authors since then (for example, \cite{VANDERLEERDURAN2018, LINTOLMAN2006}). It can also be proven that these generalized structures are equivalent to a particular structure defined on the manifold.

It is worth noting that in this paper we do not study the integrability of any structure. Therefore, each polynomial structure is preceded by the adverb ``almost'' (e.g. generalized almost complex structures).

\section{Preliminaries}
\label{SECTION:PRELIMINARIES}

\subsection{Geometric structures in vector bundles}

To introduce the relevant definitions, we work with any base manifold $M$ and a vector bundle $E\to M$, with $E_p\subset E$ denoting the fibre at $p\in M$. We define various geometric structures on $E$ as morphisms involving the $\fm$-module of sections of $E$, $\GE$, where $\fm$ is the ring of differentiable functions on $M$. We follow the terminology from specialised texts such as \cite{POOR1981}.

\begin{definition}[{\cite[Def. 3.1]{POOR1981}}]
\label{DEFINITION:METRICVECTORBUNDLE}
A \emph{metric} on the vector bundle $E$ is a differentiable morphism $g\colon \GE\times \GE\to \fm$ that is bilinear, symmetric and nondegenerate.
\end{definition}

\begin{definition}[{\cite[Def. 8.4]{POOR1981}}]
A \emph{symplectic structure} on the vector bundle $E$ is defined as a differentiable morphism $\omega\colon \Gamma(E)\times \Gamma(E)\to \mathfrak F(M)$ that is bilinear, skew-symmetric and nondegenerate.
\end{definition}

If a metric is positive definite it is named \emph{Riemannian metric}, whilst if it is not positive definite it is called \emph{pseudo-Riemannian metric}. For a pseudo-Riemannian metric, we can define its \emph{signature} $(r, s)$ as the signature of the metric $g_p\colon E_p\times E_p\to \R$ that is obtained restricting $g$ to each fibre $E_p$ as a vector space. When $r = s$, the metric signature is said to be \emph{neutral}.

In the case of a metric, if we work with the tangent bundle $E = TM$, we recover the usual concepts of Riemannian manifold or pseudo-Riemannian manifold $(M, g)$. If we work with the generalized tangent bundle $E = \TM$, we will talk about generalized metrics. Although the definition of a generalized metric by M. Gualtieri in \cite[Def. 1.9]{GUALTIERI2014} is more restrictive, in this paper we follow the terminology from Definition \ref{DEFINITION:METRICVECTORBUNDLE} taking the specific case $E = \TM$ in order to study more examples.

For a symplectic structure, taking the particular case $E = TM$ we retrieve the concept of almost symplectic manifold $(M, \omega)$ (an almost symplectic manifold is called a symplectic manifold when $d\omega = 0$). In the case $E = \TM$, these morphisms will be called generalized symplectic structures.

Using a metric or a symplectic structure on a vector bundle we can obtain two isomorphisms between $\GE$ and $\Gamma(E^*)$, known as \emph{musical isomorphisms}. Using a symplectic structure $\omega$ (it is analogous for a metric $g$), these morphisms are the following.

\begin{definition}[{\cite[Def. 3.8]{POOR1981}}]
The \emph{flat isomorphism} $\flat_{\omega}\colon \GE\to \Gamma(E^*)$ associated to $\omega$ takes each $X\in \GE$ to the section $\flat_{\omega}X\in \Gamma(E^*)$, defined as $(\flat_{\omega}X)(Y) := \omega(X, Y)$ for every $Y\in \GE$. The \emph{sharp isomorphism} associated to $\omega$, $\sharp_{\omega}\colon \Gamma(E^*)\to \GE$, is the inverse of $\flat_{\omega}$.
\end{definition}

Other structures that can be defined on any vector bundle $E\to M$ are known as polynomial structures. These kind of structures are specially relevant in generalized geometry.

\begin{definition}
\label{DEFINITION:POLYNOMIALSTRUCTUREVECTORBUNDLE}
A \emph{polynomial structure} on the vector bundle $E$ is defined as a morphism $J\colon \GE\to \GE$ with a minimal polynomial $P$ such that $P(J) = 0$. When $P = x^2 + 1$ (that is, $J^2 = -Id$) $J$ is called an \emph{almost complex structure} \cite[Def. 1.58]{POOR1981}; if $P = x^2 - 1$ (in other words, $J^2 = +Id$) $J$ is an \emph{almost product structure}. When the $+1, -1$-eigenbundles of an almost product structure have the same dimension, it is called an \emph{almost paracomplex structure}.
\end{definition}

When we work with the tangent bundle $E = TM$, the usual concepts of almost complex, almost product and almost paracomplex manifolds $(M, J)$ arise. When we take the generalized tangent bundle $E = \TM$, we call these morphisms generalized almost complex, generalized almost product and generalized almost paracomplex structures. As in the case of the generalized metrics, the definition of a generalized almost complex structure by M. Gualtieri in \cite[Def. 3.1]{GUALTIERI2011} and of a generalized almost paracomplex structure by A. Wade in \cite[Def. 2.2]{WADE2004} are more restrictive. Therefore, in order to study other examples that are also interesting, we follow the terminology given in Definition \ref{DEFINITION:POLYNOMIALSTRUCTUREVECTORBUNDLE} for the specific bundle $E = \TM$.

It is worth noting that a polynomial structure $J$ on a vector bundle $E$ induces a geometric structure $J^*$ on the dual vector bundle $E^*$. This \emph{dual structure} is the endomorphism $J^*\colon \Gamma(E^*)\to \Gamma(E^*)$ such that for each $\xi\in \Gamma(E^*)$ the section $J^*\xi$ is defined as $(J^*\xi)(X) := \xi(JX)$ for every $X\in \GE$. It is immediate to see that the minimal polynomial of $J^*$ is the same that the one associated to $J$.

In this paper, we are going to focus in the interaction between different geometric structures. First, we study the possible interactions between a metric $g$ and an almost complex or almost product structure $J$. If this endomorphism interacts with the metric $g$ as an isometry or anti-isometry, we reach the following definition.

\begin{definition}
An \emph{$(\al, \var)$-metric structure} on the vector bundle $E$, with $\al, \var\in \{+1, -1\}$, is defined as a structure $(J, g)$ formed by a polynomial structure $J\colon \GE\to \GE$ and a metric $g\colon \GE\times \GE\to \fm$ such that
	\begin{equation*}
	  J^2 = \al Id, \enspace\enspace\enspace g(JX, JY) = \var g(X, Y),
	\end{equation*}
for every $X, Y\in \GE$. Depending on the values of $\al$ and $\var$, the structure receives a different name:

\begin{itemize}
  
\item If $\al = +1$, $\var = +1$, the structure $(J, g)$ is called \emph{Riemannian} or \emph{pseudo-Riemannian almost product}, depending on wether $g$ is a Riemannian or a pseudo-Riemannian metric. If $g$ is Riemannian and $J$ is an almost paracomplex structure, then $(J, g)$ is an \emph{almost para-Norden structure}.
  
\item If $\al = +1$, $\var = -1$, the structure $(J, g)$ is called \emph{almost para-Hermitian}. Because of the compatibility condition between $J$ and $g$, the metric must be pseudo-Riemannian with neutral signature and $J$ must be an almost paracomplex structure.
  
\item If $\al = -1$, $\var = +1$, the structure $(J, g)$ is called \emph{almost Hermitian} or \emph{indefinite almost Hermitian}, depending on wether $g$ is Riemannian or pseudo-Riemannian (the second term comes from \cite{BARROSROMERO1982}).
  
\item If $\al = -1$, $\var = -1$, the structure $(J, g)$ is called \emph{almost Norden}. As in the case of the almost para-Hermitian structures, the metric $g$ must be pseudo-Riemannian with neutral signature.
    
\end{itemize}
\end{definition}

It is easy to see that the condition $g(JX, JY) = \var g(X, Y)$ is equivalent to $g(JX, Y) = \al\var g(X, JY)$. Because of this, there is an important geometric structure that can be defined using the metric $g$ and the endomorphism $J$. We call this structure the \emph{fundamental tensor} associated to $(J, g)$, and it is defined as $\phi\colon \GE\times \GE\to \fm$ such that
	\begin{equation*}
	  \phi(X, Y) := g(JX, Y),
	\end{equation*}
for every $X, Y\in \GE$. It can be easily checked that when $\al\var = +1$ the morphism $\phi$ is a metric on $E$ (called the \emph{twin metric}), whereas when $\al\var = -1$ it is a symplectic structure on the vector bundle (called the \emph{fundamental symplectic structure}). The musical isomorphisms of the structure $\phi$ are related to the musical isomorphisms of $g$ via the endomorphism $J$, as it is stated in the following proposition.

\begin{proposition}[{\cite[Prop. 2.8]{ETAYOGOMEZNICOLASSANTAMARIA2022}}]
\label{PROPOSITION:EQUALITIESFLATSHARP}
Let $(J, g)$ be an $(\al, \var)$-metric structure on the vector bundle $E$ with fundamental tensor $\phi$. Then, the following equalities hold:
	\begin{equation}
	  \begin{gathered}
	    \flat_{\phi} = \flat_g J = \al\var J^*\flat_g,  \\
	    \var \sharp_{\phi} = \sharp_g J^* = \al\var J\sharp_g.
	  \end{gathered}
	\label{EQUATION:EQUALITIESFLATSHARP}
	\end{equation}
\end{proposition}

The term of $(\al, \var)$-metric structure comes from \cite{ETAYOSANTAMARIA2000, ETAYOSANTAMARIA2016}, where a lot of properties have been studied (using Riemannian metrics when $\var = +1$) on the bundle $E = TM$. In this case, we use the notation $(M, J, g)$ for an $(\al, \var)$-metric manifold. When the vector bundle considered is $E = \TM$, these structures will be called generalized $(\al, \var)$-metric structures.

The last structures that are defined in this section arise when two polynomial structures on a vector bundle commute or anti-commute. Then, we can define the following geometric structures.

\begin{definition}
\label{DEFINITION:TRIPLESTRUCTUREVECTORBUNDLE}
A \emph{triple structure} on a vector bundle $E$ can be defined as the structure $(F, J, K)$ composed by three polynomial structures $F, J, K\colon \GE\to \GE$ such that they are almost complex or almost product (not necessarily the same type), $F$ and $J$ commute or anti-commute (that is, $JF = \lambda FJ$ with $\lambda\in \{+1, -1\}$), and $K = FJ$. There are four types of triple structures:

\begin{itemize}

\item \emph{Almost hypercomplex structures}. In this case, $F, J, K$ are almost complex structures that anti-commute, that is, $F^2 = J^2 = K^2 = FJK = -Id$.
  
\item \emph{Almost bicomplex structures}. In this case, $F, J$ are almost complex structures and $K$ is an almost product structure, such that $F, J, K$ commute. In other words, $F^2 = J^2 = -K^2 = -FJK = -Id$.
    
\item \emph{Almost biparacomplex structures}. In this case, $F, J$ are almost product structures (it can be proven that in this case $F, J$ are almost paracomplex structures) and $K$ is an almost complex structure, such that $F, J, K$ anti-commute. In other words, $F^2 = J^2 = -K^2 = -FJK = +Id$.
  
\item \emph{Almost hyperproduct structures}. In this case, $F, J, K$ are almost product structures that commute, that is, $F^2 = J^2 = K^2 = FJK = +Id$.

\end{itemize}
\end{definition}

Naturally, the endomorphisms that conform a triple structure can interact with a metric $g$, generating isometric or anti-isometric structures that can behave as $(\al, \var)$-metric structures on the vector bundle $E$. The behaviour of a triple structure with a metric is included in the next statement, that is straightforward to check.

\begin{proposition}
Let $(F, J, K)$ be a triple structure on the vector bundle $E$ and $g$ a metric on $E$ such that there are $\var_1 \var_2\in \{+1, -1\}$ with $g(FX, FY) = \var_1 g(X, Y)$ and $g(JX, JY) = \var_2 g(X, Y)$ for each $X, Y\in \GE$. Then, we have that $g(KX, KY) = \var_1\var_2 g(X, Y)$. In particular, if $F$ is an $(\al_1, \var_1)$-metric structure and $J$ is an $(\al_2, \var_2)$-metric structure on the vector bundle $E$ both related to a suitable metric $g$, and $JF = \lambda FJ$ with $\lambda\in \{+1, -1\}$, then $K$ is a $(\lambda\al_1\al_2, \var_1\var_2)$-metric structure.
\end{proposition}

These structures have been studied, for example, in \cite{HSU1960, ETAYOSANTAMARIATRIAS2004} for the tangent bundle $E = TM$. For the generalized tangent bundle $E = \TM$, these generalized triple structures will be called generalized almost hypercomplex, generalized almost bicomplex, generalized almost biparacomplex and generalized almost hyperproduct.

\subsection{Generalized geometry}

From now on, we work with the \emph{generalized tangent bundle}, defined as $\TM := \TTM\to M$. As we have indicated before, we want to specialise the definitions given for any vector bundle $E\to M$ to the case $E = \TM$, working with sections $\GTM = \xm\oplus \dm$, with $\xm$ as the $\fm$-module of vector fields on $M$ and $\dm$ as the $\fm$-module of 1-forms on $M$.

It is important to introduce the matrix notation that will be used throughout this document. Any section endomorphism $\K\colon \GTM\to \GTM$ can be written as
	\begin{equation}
	  \mathcal K =
	  \left(
	    \begin{array}{cc}
		  H     & \sigma  \\
		  \tau  & K
		\end{array}
	  \right),
	  \label{EQUATION:MATRIXNOTATION}
	\end{equation}
where $H\colon \mathfrak X(M)\to \mathfrak X(M), \enspace \sigma\colon \Lambda^1(M)\to \mathfrak X(M), \enspace \tau\colon \mathfrak X(M)\to \Lambda^1(M), \enspace K\colon \Lambda^1(M)\to \Lambda^1(M)$. This means that for any $X + \xi\in \Gamma(\mathbb TM)$ we can write
	\begin{equation*}
	  \mathcal K(X + \xi) = \left(
	    \begin{array}{cc}
	      H     & \sigma  \\
	      \tau  & K
	    \end{array}
	  \right) \left(
	    \begin{array}{c}
	      X    \\
	      \xi
	    \end{array}
	  \right) = \left(
	    \begin{array}{c}
	      HX + \sigma \xi  \\
	      \tau X + K\xi
	    \end{array}
	  \right) = (HX + \sigma \xi) + (\tau X + K\xi).
	\end{equation*}

To end this section, we present three \emph{canonical structures} that emerge on $\TM$ without the need of adding any additional structure to $M$. These structures are the \emph{natural generalized metric} $\G_0$, the \emph{natural generalized symplectic structure} $\Omega_0$ and the \emph{natural generalized almost paracomplex structure} $\F_0$. They are defined as follows.

\begin{definition}[{\cite[Sec. 1]{GUALTIERI2011}}]
\label{DEFINITION:CANONICALMETRIC}
The \emph{natural generalized metric} over $\TM$ is defined as $\G_0\colon \GTM\times \GTM\to \fm$ such that for each $X, Y\in \xm$ and $\xi, \eta\in \dm$,
    \begin{equation}
	  \mathcal G_0(X + \xi, Y + \eta) := \frac{1}{2}(\xi(Y) + \eta(X)).
	  \label{EQUATION:CANONICALMETRIC}
	\end{equation}
\end{definition}

\begin{definition}[{\cite[Sec. 2]{NANNICINI2006}}]
The \emph{natural generalized symplectic structure} over $\TM$ is defined as the structure $\Omega_0\colon \GTM\times \GTM\to \fm$ such that for each $X, Y\in \xm$ and $\xi, \eta\in \dm$,
    \begin{equation}
	  \Omega_0(X + \xi, Y + \eta) := \frac{1}{2}(\xi(Y) - \eta(X)).
	  \label{EQUATION:CANONICALSYMPLECTICSTRUCTURE}
	\end{equation}
\end{definition}

\begin{definition}[{\cite[Ex. 1]{WADE2004}}]
\label{DEFINITION:CANONICALALMOSTPARACOMPLEXSTRUCTURE}
The \emph{natural generalized almost paracomplex structure} is defined as the polynomial structure $\F_0\colon \GTM\to \GTM$ such that $\F_0(X + \xi) := -X + \xi$ for all $X\in \xm$ and $\xi\in \dm$. In matrix notation,
	\begin{equation}
	  \F_0 =
	  \left(
	    \begin{array}{cc}
		  -Id  & 0   \\
		  0    & Id
		\end{array}
	  \right).
	  \label{EQUATION:CANONICALALMOSTPARACOMPLEXSTRUCTURE}
	\end{equation}
\end{definition}

\section{Generalized metrics}
\label{SECTION:GENERALIZEDMETRICS}

Before studying the different $(\al, \var)$-metric structures that can be constructed on $\TM$, it is worthwhile to analyse the different generalized metrics that can be defined. The generalized tangent bundle has a rich structure; in particular, the natural generalized metric $\G_0$ defined in Eq. \eqref{EQUATION:CANONICALMETRIC} arises without adding any additional structure to the manifold $M$. Therefore, it is interesting to ask if any generalized metric is related to $\G_0$. The following proposition, based on an observation from \cite[Sec. 6.2]{GUALTIERI2004}, shows the answer to that question.

\begin{proposition}
\label{PROPOSITION:CHARACTERIZATIONGENERALIZEDMETRICS}
Any generalized metric $\G\colon \GTM\times \GTM\to \fm$ can be obtained from the natural generalized metric $\G_0$ and an injective endomorphism $\K\colon \GTM\to \GTM$, with
	\begin{equation}
	  \G(X + \xi, Y + \eta) = \G_0(\K(X + \xi), Y + \eta),
	  \label{EQUATION:CHARACTERIZATIONGENERALIZEDMETRICS}
	\end{equation}
for every $X + \xi, Y + \eta\in \GTM$. In matrix notation, $\K$ has the form
	\begin{equation*}
	  \K =
	  \left(
	    \begin{array}{cc}
		  H        & \sigma  \\
		  \tau  & H^*
		\end{array}
	  \right),
	\end{equation*}
such that $(\tau X)(Y) = (\tau Y)(X)$ and $\xi(\sigma\eta) = \eta(\sigma\xi)$ for all $X, Y\in \xm$ and $\xi, \eta\in \dm$.

Reciprocally, each injective endomorphism $\K$ satisfying the above mentioned conditions induces a generalized metric $\G$ defined as in Eq. \eqref{EQUATION:CHARACTERIZATIONGENERALIZEDMETRICS}.
\end{proposition}

\begin{proof}
First, from a generalized metric $\G\colon \GTM\times \GTM\to \fm$ we can get the morphisms $\tau\colon \xm\to \dm$, $\sigma\colon \dm\to \xm$, $H\colon \xm\to \xm$, defined as
	\begin{equation*}
      (\tau X)(Y) = 2\G(X, Y), \enspace\enspace\enspace \eta(\sigma \xi) = 2\G(\xi, \eta), \enspace\enspace\enspace \xi(HX) = 2\G(X, \xi),
	\end{equation*}
for every $X, Y\in \xm$ and $\xi, \eta\in \dm$. It is straightforward to see that $(\tau X)(Y) = (\tau Y)(X)$ and $\xi(\sigma\eta) = \eta(\sigma\xi)$. Then, the endomorphism $\K$ obtained for such $\tau, \sigma, H$ fulfils the wanted relation:
	\begin{align*}
	  \G_0(\K(X + \xi), Y + \eta) &= \G_0(HX + \sigma\xi + \tau X + H^*\xi, Y + \eta) = \frac{1}{2}((\tau X + H^*\xi)(Y) + \eta(HX + \sigma\xi))  \\
	  							  &= \frac{1}{2}((\tau X)(Y) + \xi(HY) + \eta(HX) + \eta(\sigma\xi)) = \G(X, Y) + \G(Y, \xi) + \G(X, \eta) + \G(\xi, \eta)  \\
	  							  &= \G(X + \xi, Y + \eta).
	\end{align*}
To see that $\K$ is injective, if $\K(X + \xi) = 0$ then $0 = \G_0(\K(X + \xi), Y + \eta) = \G(X + \xi, Y + \eta)$. Therefore, as $\G$ is a non-degenerate metric then $X + \xi = 0$ and $\K$ is injective.
	
Reciprocally, it is immediate to see that the morphism $\G$ defined as in Eq. \eqref{EQUATION:CHARACTERIZATIONGENERALIZEDMETRICS} for a suitable endomorphism $\K$ is bilinear. We check that it is symmetric:
	\begin{gather*}
	    \G(X, Y) = \G_0(HX + \tau X, Y) = \frac{1}{2}(\tau X)(Y) = \frac{1}{2}(\tau Y)(X) = \G_0(HY + \tau Y, X) = \G(Y, X),  \\
	    \G(\xi, \eta) = \G_0(\sigma\xi + H^*\xi, \eta) = \frac{1}{2}\eta(\sigma\xi) = \frac{1}{2}\xi(\sigma\eta) = \G_0(\sigma\eta + H^*\eta, \xi) = \G(\eta, \xi),  \\
	    \G(X, \xi)    = \G_0(HX + \tau X, \xi) = \frac{1}{2}\xi(HX) = \frac{1}{2}(H^*\xi)(X) = \G_0(\sigma\xi + H^*\xi, X) = \G(\xi, X).
	\end{gather*}
Finally, to check that the generalized metric $\G$ is nondegenerate we assume that there is a $X + \xi\in \GTM$ such that $\G(X + \xi, Y + \eta) = \G_0(\K(X + \xi), Y + \eta) = 0$ for all $Y + \eta\in \GTM$. Then, as $\G_0$ is nondegenerate, it must be $\K(X + \xi) = 0$. The morphism $\K$ is injective, so $X + \xi = 0$ and $\G$ is nondegenerate.
\end{proof}

An analogous result can be stated for generalized symplectic structures. The proof of this proposition is parallel to the proof of Proposition \ref{PROPOSITION:CHARACTERIZATIONGENERALIZEDMETRICS} and therefore it will be omitted.

\begin{proposition}
Any generalized symplectic structure $\ \Omega\colon \GTM\times \GTM\to \fm$ can be obtained from the natural generalized metric $\G_0$ and an injective endomorphism $\K\colon \GTM\to \GTM$, such that
	\begin{equation}
	  \Omega(X + \xi, Y + \eta) = \G_0(\K(X + \xi), Y + \eta),
	  \label{EQUATION:CHARACTERIZATIONGENERALIZEDSYMPLECTIC}
	\end{equation}
for every $X + \xi, Y + \eta\in \GTM$. In matrix notation, $\K$ has the form
	\begin{equation*}
	  \K =
	  \left(
	    \begin{array}{cc}
		  H        & \sigma  \\
		  \tau  & -H^*
		\end{array}
	  \right),
	\end{equation*}
such that $(\tau X)(Y) = -(\tau Y)(X)$ and $\xi(\sigma\eta) = -\eta(\sigma\xi)$ for all $X, Y\in \xm$ and $\xi, \eta\in \dm$.

Reciprocally, each injective endomorphism $\K$ satisfying the above mentioned conditions induces a generalized symplectic structure $\Omega$ defined as in Eq. \eqref{EQUATION:CHARACTERIZATIONGENERALIZEDSYMPLECTIC}. 
\end{proposition}

There are some endomorphisms $\K\colon \GTM\to \GTM$ inducing generalized metrics and generalized symplectic structures that are especially interesting.

\begin{example}
When we have an injective endomorphism $H\colon \xm\to \xm$, the endomorphism
	\begin{equation}
	  \label{EQUATION:DIAGONALEXAMPLEMETRIC}
	  \K_{\lambda, H} =
	  \left(
	    \begin{array}{cc}
		  H  & 0  \\
		  0  & \lambda H^*
		\end{array}
	  \right),
	\end{equation}
induces a generalized metric when $\lambda = +1$ and a generalized symplectic structure when $\lambda = -1$. This happens, for example, when $H$ is an almost complex or an almost product structure on the manifold.
\end{example}

\begin{example}
\label{EXAMPLE:ANTIDIAGONALEXAMPLEMETRIC}
When we have a metric $g$ or an almost symplectic structure $\omega$ defined on $M$, their musical isomorphisms can be used to induce a generalized metric and a generalized symplectic structure via the following injective endomorphisms, respectively:
	\begin{equation}
	  \label{EQUATION:ANTIDIAGONALEXAMPLEMETRIC}
	  \F_g =
	  \left(
	    \begin{array}{cc}
		  0        & \sharp_g  \\
		  \flat_g  & 0
		\end{array}
	  \right), \enspace\enspace\enspace
	  \F_{\omega} =
	  \left(
	    \begin{array}{cc}
		  0               & \sharp_{\omega}  \\
		  \flat_{\omega}  & 0
		\end{array}
	  \right).
	\end{equation}
\end{example}

\begin{example}
In \cite[Sec. 3]{NANNICINI2019}, A. Nannicini induces a generalized metric from a Norden manifold $(M, J, g)$. This generalized metric $\G\colon \GTM\times \GTM\to \fm$ is defined as follows:
	\begin{equation}
	  \label{EQUATION:NANNICINIEXAMPLEMETRIC}
	  \G(X + \xi, Y + \eta) = g(X, Y) + \frac{1}{2}g(JX, \sharp_g\eta) + \frac{1}{2}g(\sharp_g\xi, JY) + g(\sharp_g\xi, \sharp_g\eta).
	\end{equation}
If we analyse this metric in the sense of Proposition \ref{PROPOSITION:CHARACTERIZATIONGENERALIZEDMETRICS}, it can be checked that it is induced by the endomorphism
	\begin{equation*}
	  \K = \K_{+1, J} + 2\F_g =
	  \left(
	    \begin{array}{cc}
		  J         & 2\sharp_g  \\
		  2\flat_g  & J^*
		\end{array}
	  \right),
	\end{equation*}
using the notation from Eqs. (\ref{EQUATION:DIAGONALEXAMPLEMETRIC}, \ref{EQUATION:ANTIDIAGONALEXAMPLEMETRIC}) for the almost complex structure $J$ and the metric $g$. In effect,
	\begin{align*}
	  \G_0(\K(X + \xi), Y + \eta) &= \G_0(JX + 2\sharp_g \xi + 2\flat_g X + J^*\xi, Y + \eta) = \frac{1}{2}((2\flat_g X + J^*\xi)(Y) + \eta(JX + 2\sharp_g \xi))  \\
	                              &= (\flat_g X)(Y) + \frac{1}{2}\xi(JY) + \frac{1}{2}\eta(JX) + \eta(\sharp_g\xi)  \\
	                              &= g(X, Y) + \frac{1}{2}g(JX, \sharp_g\eta) + \frac{1}{2}g(\sharp_g\xi, JY) + g(\sharp_g\xi, \sharp_g\eta).
	\end{align*}
\end{example}

Example \ref{EXAMPLE:ANTIDIAGONALEXAMPLEMETRIC} gives the most straightforward way to induce a generalized metric using a metric $g$ defined on a manifold, and a generalized symplectic structure from an almost symplectic structure $\omega$ on a manifold. Therefore, we present the following definition and properties.

\begin{definition}
\label{DEFINITION:INDUCEDGENERALIZEDMETRIC}
Let $(M, g)$ be a Riemannian or pseudo-Riemannian manifold. Then, we define the \emph{generalized metric induced by $g$}, $\G_g\colon \GTM\times \GTM\to \fm$, as
	\begin{equation}
	  \G_g(X + \xi, Y + \eta) := \G_0(2\F_g(X + \xi), Y + \eta) = g(X, Y) + g(\sharp_g \xi, \sharp_g \eta).
	  \label{EQUATION:INDUCEDGENERALIZEDMETRIC}
	\end{equation}
\end{definition}

This generalized metric agrees with the definition of a metric on a dual vector bundle $E^*\to M$ and on a Whitney sum $E\oplus F\to M$ given in \cite[Defs. 3.6, 3.9]{POOR1981}. Also, this metric has been used in some previous works such as \cite{HULLLINDSTROM2020}. 

There are some easy properties to check. For example, it is immediate to see that $TM$ and $T^*M$ are orthogonal distributions of $\TM$ with the metric $\G_g$. With respect to the signature of $\G_g$, we make the following affirmation.

\begin{proposition}
If $(M, g)$ is a pseudo-Riemannian manifold with signature $(r, s)$, then $\G_g$ is a generalized pseudo-Riemannian metric with signature $(2r, 2s)$. In particular, $\G_g$ is a generalized Riemannian metric if $(M, g)$ is a Riemannian manifold, and $\G_g$ is a generalized pseudo-Riemannian metric with neutral signature if $(M, g)$ is a pseudo-Riemannian manifold with neutral signature.
\end{proposition}

\begin{proof}
We fix a $p\in M$ and work with the spaces $T_pM$, $T_p^*M$, $\mathbb T_pM$. We take a basis $\{v_1, \dots, v_{r+s}\}\subset T_pM$ such that $g(v_i, v_i) > 0$ for $i = 1, \dots, r$ and $g(v_j, v_j) < 0$ for $j = r+1, \dots, r+s$. As we can use the $\flat_g$ isomorphism to obtain a basis of $T_p^*M$, we compute $\G_g(v_i, v_i)$ and $\G_g(\flat_g v_i, \flat_g v_i)$: if $i\in \{1, \dots, r\}$, then
	\begin{equation*}
	  \G_g(v_i, v_i) = g(v_i, v_i) > 0, \enspace\enspace\enspace \G_g(\flat_g v_i, \flat_g v_i) = g(\sharp_g\flat_g v_i, \sharp_g\flat_g v_i) = g(v_i, v_i) > 0,
	\end{equation*}
whilst for $j\in \{r+1, \dots, r+s\}$ we have
	\begin{equation*}
	  \G_g(v_j, v_j) = g(v_j, v_j) < 0, \enspace\enspace\enspace \G_g(\flat_g v_j, \flat_g v_j) = g(\sharp_g\flat_g v_j, \sharp_g\flat_g v_j) = g(v_j, v_j) < 0.
	\end{equation*}
Therefore, as $\{v_1, \flat_g v_1, \dots, v_{r+s}, \flat_g v_{r+s}\}$ is a basis of $\mathbb T_pM$, the signature of the generalized metric $\G_g$ is $(2r, 2s)$.
\end{proof}

\section{Generalized $(\al, \var)$-metric structures}
\label{SECTION:GENERALIZEDALVARMETRICSTRUCTURES}

In \cite{ETAYOGOMEZNICOLASSANTAMARIA2022}, we studied various generalized polynomial structures induced by different geometric structures on the base manifold. In this section, we study the compatibility of those generalized polynomial structures with different generalized metrics. With those endomorphisms, we construct $(\al, \var)$-metric structures on $\TM$, as it was done in Section \ref{SECTION:PRELIMINARIES} for any vector bundle $E\to M$.

\subsection{Natural generalized $(\al, \var)$-metric structures}

First, we use the natural generalized metric $\G_0$ from Definition \ref{DEFINITION:CANONICALMETRIC} and the generalized almost paracomplex structure $\F_0$ from Definition \ref{DEFINITION:CANONICALALMOSTPARACOMPLEXSTRUCTURE} in order to find generalized $(\al, \var)$-metric structures. If any of these structures appears in the definition of the $(\al, \var)$-metric structure, we use the following terminology.

\begin{definition}
A generalized $(\al, \var)$-metric structure is called \emph{natural} if any of the canonical structures is involved directly in its definition, that is, if the generalized metric considered is $\G_0$ or the polynomial structure is $\F_0$.
\end{definition}

First, we analyse the possible interactions between the natural generalized almost paracomplex structure $\F_0$ and any generalized metric $\G$. In this situation, the following result can be inferred.

\begin{proposition}
Let $\ \G$ be any generalized metric. If $\ \G$ is a Riemannian metric, then $(\F_0, \G)$ is a natural generalized almost para-Norden structure if and only if $\ \G(X, \xi) = 0$ for all $X\in \xm$, $\xi\in \dm$. On the other hand, if $\G$ is a pseudo-Riemannian metric, then $(\F_0, \G)$ is a natural generalized pseudo-Riemannian almost product structure if and only if $\ \G(X, \xi) = 0$ for all $X\in \xm$, $\xi\in \dm$, while $(\F_0, \G)$ is an almost para-Hermitian structure if and only if $\ \G(X, Y) = \G(\xi, \eta) = 0$ for every $X, Y\in \xm$ and $\xi, \eta\in \dm$.
\end{proposition}

\begin{proof}
A straightforward calculation shows that
	\begin{align*}
	  \G(\F_0(X + \xi), \F_0(Y + \eta)) &= \G(-X + \xi, -Y + \eta) = \G(-X, -Y) + \G(-X, \eta) + \G(\xi, -Y) + \G(\xi, \eta)  \\
	  									&= (\G(X, Y) + \G(\xi, \eta)) - (\G(X, \eta) + \G(Y, \xi)).
	\end{align*}
Therefore, if $\G(X, \xi) = 0$ the structure $(\F_0, \G)$ is generalized almost para-Norden or pseudo-Riemannian almost product, and if $\G(X, Y) = \G(\xi, \eta) = 0$ the structure $(\F_0, \G)$ is generalized almost para-Hermitian. To prove the reverse, if $(\F_0, \G)$ is a generalized $(\al, \var)$-metric structure, then
	\begin{gather*}
	  \var \G(X, \xi) = \G(\F_0 X, \F_0 \xi) = -\G(X, \xi),  \\
	  \var \G(X, Y) = \G(\F_0 X, \F_0 Y) = \G(X, Y),  \\
	  \var \G(\xi, \eta) = \G(\F_0 \xi, \F_0 \eta) = \G(\xi, \eta),
	\end{gather*}
and the result is proven.
\end{proof}

As the natural generalized metric meets the equality $\G_0(X, Y) = \G_0(\xi, \eta) = 0$, we obtain the following corollary that links the three canonical generalized structures on $\TM$.

\begin{corollary}
The pair $(\F_0, \G_0)$ is a natural generalized almost para-Hermitian structure with the canonical structure $\Omega_0$ as fundamental symplectic structure. In other words, the three canonical generalized structures from Eqs. \emph{(\ref{EQUATION:CANONICALMETRIC}, \ref{EQUATION:CANONICALSYMPLECTICSTRUCTURE}, \ref{EQUATION:CANONICALALMOSTPARACOMPLEXSTRUCTURE})} are related between them with the expression
	\begin{equation*}
	  \Omega_0(X + \xi, Y + \eta) = \G_0(\F_0(X + \xi), Y + \eta).
	\end{equation*}
\end{corollary}

Now we analyse the natural generalized $(\al, \var)$-metric structures involving the natural generalized metric $\G_0$. As $\G_0$ is a pseudo-Riemannian metric with neutral signature, there are four possible kinds of natural structures involving $\ \G_0$: generalized pseudo-Riemannian almost product structures, generalized almost para-Hermitian structures, generalized indefinite almost Hermitian structures and generalized almost Norden structures.

Originally, generalized almost complex structures are defined in \cite{HITCHIN2003, GUALTIERI2004} as endomorphisms $\J\colon \GTM\to \GTM$ that are isometric with respect to $\G_0$ such that $\J^2 = - \mathcal Id$. Therefore, the original definition corresponds to our definition of natural generalized indefinite almost Hermitian structures. Similarly, generalized almost paracomplex structures are defined in \cite{WADE2004} as endomorphisms $\F\colon \GTM\to \GTM$ that are anti-isometric with respect to $\G_0$ such that $\F^2 = + \mathcal Id$. Therefore, the original definition corresponds to our definition of natural generalized almost para-Hermitian structures. Most of these structures have been widely studied; for example, in \cite{IDAMANEA2017} both generalized almost para-Hermitian and almost para-Norden structures are analysed.

Isometric generalized almost complex structures with respect to $\G_0$ are particularly important. For example, we present now a result that relates the existence of natural generalized indefinite almost Hermitian structures with the existence of almost complex structures on the manifold $M$.

\begin{proposition}[{\cite[Sec. 6.4]{GUALTIERI2004}}]
Let $M$ be a manifold. Then, there exists a natural generalized indefinite almost Hermitian structure $(\J, \G_0)$ if and only if there exists an almost complex structure $J$ on $M$.
\end{proposition}

Starting from a natural generalized indefinite almost Hermitian structure $(\J, \G_0)$, the idea of the proof is to construct a positive definite subbundle $L_+\subset \TM$ that is stable with respect to $\J$. This subbundle can be constructed in such a way that the projection $\pi\colon \TM\to TM$ defined as $\pi(X + \xi) = X$ is an isomorphism between $L_+$ and $TM$. Then, we can define the almost complex structure $J = \pi|_{L_+}\circ \J\circ \pi|_{L_+}^{-1}$ on the base manifold.

We analyse now the different examples of natural generalized $(\al, \var)$-metric structures that are obtained when the base manifold is endowed with a particular geometry. Taking a Riemannian or pseudo-Riemannian manifold $(M, g)$, the musical isomorphisms of $g$ induce a generalized almost complex and a generalized almost paracomplex structure. These induced polynomial structures on $\TM$ are compatible with the natural generalized metric, as it is shown in the following proposition.

\begin{proposition}[{{\cite[Ex. 3.1]{IDAMANEA2017}}, {\cite[Sec. 2]{NANNICINI2013}}}]
\label{PROPOSITION:NATURALALVARMETRIC}
Let $(M, g)$ be a Riemannian or pseudo-Riemannian manifold. Then, the generalized polynomial structures
    \begin{gather}
      \label{EQUATION:COMPLEXG}
      \J_g =
      \left(
        \begin{array}{cc}
          0        & -\sharp_g  \\
          \flat_g  & 0
        \end{array}
      \right),  \\
      \label{EQUATION:PRODUCTG}
      \F_g =
      \left(
        \begin{array}{cc}
          0        & \sharp_g  \\
          \flat_g  & 0
        \end{array}
      \right),
    \end{gather}
generate, respectively, a natural generalized almost Norden structure $(\J_g, \G_0)$ and a natural generalized pseudo-Riemannian almost product structure $(\F_g, \G_0)$. Their twin metrics are, respectively,
	\begin{gather*}
	  \Phi_{\J_g}(X + \xi, Y + \eta) = \frac{1}{2}(g(X, Y) - g(\sharp_g \xi, \sharp_g \eta)),  \\
	  \Phi_{\F_g}(X + \xi, Y + \eta) = \frac{1}{2}(g(X, Y) + g(\sharp_g \xi, \sharp_g \eta)).
	\end{gather*}
\end{proposition}

Similarly, when we have an almost symplectic manifold $(M, \omega)$ the musical isomorphisms $\flat_{\omega}, \sharp_{\omega}$ induce generalized polynomial structures, namely, a generalized almost complex and a generalized almost paracomplex structure. These polynomial structures form different $(\al, \var)$-metric structures on the bundle $\TM$, as it is shown in the following proposition, analogous to Proposition \ref{PROPOSITION:NATURALALVARMETRIC}.

\begin{proposition}[{{\cite[Sec. 3]{GUALTIERI2011}}, {\cite[Ex. 2]{WADE2004}}}]
\label{PROPOSITION:NATURALALVARSYMPLECTIC}
Let $(M, \omega)$ be an almost symplectic manifold. Then, the generalized polynomial structures
    \begin{gather}
      \label{EQUATION:COMPLEXOMEGA}
      \J_{\omega} =
      \left(
        \begin{array}{cc}
          0        & -\sharp_{\omega}  \\
          \flat_{\omega}  & 0
        \end{array}
      \right),  \\
      \label{EQUATION:PRODUCTOMEGA}
      \F_{\omega} =
      \left(
        \begin{array}{cc}
          0        & \sharp_{\omega}  \\
          \flat_{\omega}  & 0
        \end{array}
      \right),
    \end{gather}
form, respectively, a natural generalized indefinite almost Hermitian structure $(\J_{\omega}, \G_0)$ and a natural generalized almost para-Hermitian structure $(\F_{\omega}, \G_0)$. Their fundamental symplectic structures are, respectively,
	\begin{gather*}
	  \Phi_{\J_{\omega}}(X + \xi, Y + \eta) = \frac{1}{2}(\omega(X, Y) + \omega(\sharp_{\omega} \xi, \sharp_{\omega} \eta)),  \\
	  \Phi_{\F_{\omega}}(X + \xi, Y + \eta) = \frac{1}{2}(\omega(X, Y) - \omega(\sharp_{\omega} \xi, \sharp_{\omega} \eta)).
	\end{gather*}
\end{proposition}

When we have a polynomial structure defined on $M$, we can induce natural generalized polynomial structures with the same minimal polynomial using diagonal matrices. The following propositions show the way to do this.

\begin{proposition}
Let $(M, J)$ be an almost complex manifold. Then, the generalized almost complex structure
    \begin{equation}
      \J_{\lambda, J} =
      \left(
        \begin{array}{cc}
          J  & 0  \\
          0  & \lambda J^*
        \end{array}
      \right),
      \label{EQUATION:COMPLEXLAMBDAJ}
    \end{equation}
with $\lambda\in \{+1, -1\}$, forms a natural generalized indefinite almost Hermitian structure $(\J_{-1, J}, \G_0)$ when $\lambda = -1$, and a natural generalized almost Norden structure $(\J_{+1, J}, \G_0)$ when $\lambda = +1$. Depending on $\lambda$, its fundamental tensor is
	\begin{equation*}
	  \Phi_{\J_{\lambda, J}}(X + \xi, Y + \eta) = \G_0(JX + \xi, \lambda JY + \eta).
	\end{equation*}
\end{proposition}

\begin{proposition}[{\cite[Sec. 3.2]{IDAMANEA2017}}]
Let $(M, F)$ be an almost product manifold. Then, the generalized almost product structure
    \begin{equation}
      \F_{\lambda, F} =
      \left(
        \begin{array}{cc}
          F  & 0  \\
          0  & \lambda F^*
        \end{array}
      \right),
      \label{EQUATION:PRODUCTLAMBDAF}
    \end{equation}
with $\lambda\in \{+1, -1\}$, forms a natural generalized almost para-Hermitian structure $(\F_{-1, F}, \G_0)$ when $\lambda = -1$, and a natural generalized pseudo-Riemannian almost product structure $(\F_{+1, F}, \G_0)$ when $\lambda = +1$. Depending on $\lambda$, its fundamental tensor is
	\begin{equation*}
	  \Phi_{\F_{\lambda, F}}(X + \xi, Y + \eta) = \G_0(FX + \xi, \lambda FY + \eta).
	\end{equation*}
\end{proposition}

Now, we want to obtain $(\al, \var)$-metric structures on the generalized tangent bundle $\TM$ from an $(\al, \var)$-metric manifold. If we denote the metric on the manifold as $g$ and its fundamental tensor as $\phi$, we have already induced the natural generalized $(\al, \var)$-metric structures $(\J_g, \G_0)$, $(\F_g, \G_0)$, $(\J_{\phi}, \G_0)$, $(\F_{\phi}, \G_0)$ mentioned in Propositions \ref{PROPOSITION:NATURALALVARMETRIC}, \ref{PROPOSITION:NATURALALVARSYMPLECTIC}. Also, when the polynomial structure $J$ on the manifold is almost complex then we obtain the natural generalized $(\al, \var)$-metric structures $(\J_{+1, J}, \G_0)$, $(\J_{-1, J}, \G_0)$, while if the polynomial structure $F$ on the manifold is almost product we have the natural generalized $(\al, \var)$-metric structures $(\F_{+1, F}, \G_0)$, $(\F_{-1, F}, \G_0)$.

In \cite[Props. 4.8, 4.9]{ETAYOGOMEZNICOLASSANTAMARIA2022}, we proposed various polynomial structures that in matrix notation are represented by triangular matrices, inspired in different examples from \cite{IDAMANEA2017, NANNICINI2010}. These structures are induced by $(\al, \var)$-metric manifolds. Depending on the value of $\var$, these generalized polynomial structures are compatible with the natural generalized metric $\G_0$, as it is shown in the following propositions. As their proofs are very similar, we omit the second one.

\begin{proposition}
Let $(M, J, g)$ be an $(\al, \var)$-metric manifold with $\al = -1$. Then, the induced generalized almost complex structures 
	\begin{gather}
	  \J_{J, g, \flat} =
	  \left(
		\begin{array}{cc}
		  J        & 0         \\
		  \flat_g  & \var J^*
		\end{array}
	  \right), \enspace\enspace\enspace 
	  \J_{J, g, \sharp} =
	  \left(
		\begin{array}{cc}
		  J  & \sharp_g  \\
		  0  & \var J^*
		\end{array}
	  \right),
	  \label{EQUATION:TRIANGULARMJG}
	\end{gather}
are compatible with $\G_0$ if and only if $\var = +1$, that is, if $(M, J, g)$ is an almost Hermitian manifold or an indefinite almost Hermitian manifold (depending on $g$). In this case, each one generates a natural generalized almost Norden structure $(\J_{J, g, \flat}, \G_0)$, $(\J_{J, g, \sharp}, \G_0)$, with twin metrics
	\begin{gather*}
	  \Phi_{\J_{J, g, \flat}}(X + \xi, Y + \eta) = \G_0(JX + \xi, JY + \eta) + \frac{1}{2}g(X, Y),  \\
	  \Phi_{\J_{J, g, \sharp}}(X + \xi, Y + \eta) = \G_0(JX + \xi, JY + \eta) + \frac{1}{2}g(\sharp_g\xi, \sharp_g\eta).
	\end{gather*}
\end{proposition}

\begin{proof}
We show here the calculations for both polynomial structures, based in Proposition \ref{PROPOSITION:EQUALITIESFLATSHARP}:
	\begin{align*}
	  \G_0(\J_{J, g, \flat}(X + \xi), \J_{J, g, \flat}(Y + \eta)) 	 &= \G_0(JX + \flat_g X + \var J^*\xi, JY + \flat_g Y + \var J^*\eta)  \\
	  															  	 &= \frac{1}{2}((\flat_g X + \var J^*\xi)(JY) + (\flat_g Y + \var J^*\eta)(JX))  \\
	  															  	 &= \frac{1}{2}(g(X, JY) + \var\xi(J^2Y) + g(Y, JX) + \var\eta(J^2X))  \\
	  															  	 &= \frac{1-\var}{2}g(X, JY) - \frac{\var}{2}(\xi(Y) + \eta(X))  \\
	  															  	 &= \frac{1-\var}{2}g(X, JY) - \var\G_0(X + \xi, Y + \eta),  \\
	   \G_0(\J_{J, g, \sharp}(X + \xi), \J_{J, g, \sharp}(Y + \eta)) &= \G_0(JX + \sharp_g\xi + \var J^*\xi, JY + \sharp_g\eta + \var J^*\eta)  \\
	   																 &= \frac{\var}{2}((J^*\xi)(JY + \sharp_g\eta) + (J^*\eta)(JX + \sharp_g\xi))  \\
	  															  	 &= \frac{\var}{2}(\xi(J^2Y) + g(\sharp_g \xi, J\sharp_g\eta) + \eta(J^2X) + g(\sharp_g \eta, J\sharp_g\xi))  \\
	  															  	 &= \frac{\var-1}{2}g(\sharp_g \xi, J\sharp_g \eta) - \frac{\var}{2}(\xi(Y) + \eta(X))  \\
	  															  	 &= \frac{\var-1}{2}g(\sharp_g \xi, J\sharp_g \eta) - \var\mathcal G_0(X + \xi, Y + \eta).
    \end{align*}
Then, these two structures are compatible with $\G_0$ if and only if $\var = +1$. The calculations for the twin metrics are the following:
	\begin{gather*}
	  \Phi_{\J_{J, g, \flat}}(X + \xi, Y + \eta) = \G_0(JX + \flat_g X + J^*\xi, Y + \eta) = \frac{1}{2}((\flat_g X + J^*\xi)(Y) + \eta(JX)) = \G_0(JX + \xi, JY + \eta) + \frac{1}{2}g(X, Y),  \\
	  \Phi_{\J_{J, g, \sharp}}(X + \xi, Y + \eta) = \G_0(JX + \sharp_g\xi + J^*\xi, Y + \eta) = \frac{1}{2}((J^*\xi)(Y) + \eta(JX + \sharp_g\xi)) = \G_0(JX + \xi, JY + \eta) + \frac{1}{2}g(\sharp_g\xi, \sharp_g\eta).
	\end{gather*}
Therefore, the result holds.
\end{proof}

\begin{proposition}
Let $(M, F, g)$ be an $(\al, \var)$-metric structure on $M$ with $\al = +1$. Then, the induced generalized almost product structures 
	\begin{equation}
	  \F_{F, g, \flat} =
	  \left(
		\begin{array}{cc}
		  F        & 0         \\
		  \flat_g  & -\var F^*
		\end{array}
	  \right), \enspace\enspace\enspace 
	  \F_{F, g, \sharp} =
	  \left(
		\begin{array}{cc}
		  F  & \sharp_g  \\
		  0  & -\var F^*
		\end{array}
	  \right),
	  \label{EQUATION:TRIANGULARMFG}
	\end{equation}
are compatible with $\G_0$ if and only if $\var = -1$, that is, if $(M, F, g)$ is an almost para-Hermitian manifold. Then, we have the natural generalized pseudo-Riemannian almost product structures $(\F_{F, g, \flat}, \G_0)$, $(\F_{F, g, \sharp}, \G_0)$, with twin metrics
	\begin{gather*}
	  \Phi_{\F_{F, g, \flat}}(X + \xi, Y + \eta) = \G_0(FX + \xi, FY + \eta) + \frac{1}{2}g(X, Y),  \\
	  \Phi_{\F_{F, g, \sharp}}(X + \xi, Y + \eta) = \G_0(FX + \xi, FY + \eta) + \frac{1}{2}g(\sharp_g\xi, \sharp_g\eta).
	\end{gather*}
\end{proposition}

Finally, in \cite[Props. 4.10, 4.11]{ETAYOGOMEZNICOLASSANTAMARIA2022} we presented new examples of induced generalized polynomial structures from an $(\al, \var)$-metric manifold. In the following propositions, we show that these structures form generalized $(\al, \var)$-metric structures with the metric $\G_0$ depending on the characteristics of the $(\al, \var)$-metric structure on the manifold. We prove only the first proposition, as the proof of the second one is parallel to the first one.

\begin{proposition}
Let $(M, J, g)$ be an $(\al, \var)$-metric manifold with $\al = -1$. Then, the generalized almost paracomplex structure
	\begin{equation}
	  \F_{J, g} =
	  \left(
		\begin{array}{cc}
		  J        & \sqrt{2}\ \sharp_g  \\
		  \sqrt{2}\ \flat_g  & \var J^*
		\end{array}
	  \right),
	  \label{EQUATION:PRODUCTJG}
	\end{equation}
is compatible with $\G_0$ if and only if $\var = +1$, that is, if $(M, J, g)$ is an almost Hermitian manifold or an indefinite almost Hermitian manifold (depending on $g$). In this case, $(\F_{J, g}, \G_0)$ is a natural generalized pseudo-Riemannian almost product structure. The twin metric associated to this structure is
	\begin{equation*}
	  \Phi_{\F_{J, g}}(X + \xi, Y + \eta) = \G_0(JX + \xi, JY + \eta) + \frac{\sqrt{2}}{2}(g(X, Y) + g(\sharp_g\xi, \sharp_g\eta)).
	\end{equation*}
\end{proposition}

\begin{proof}
We use Eq. \eqref{EQUATION:EQUALITIESFLATSHARP} to check that $\F_{J, g}$ is compatible with the natural generalized metric $\G_0$:
	\begin{align*}
	  \G_0(\F_{J, g}(X + \xi), \F_{J, g}(Y + \eta)) =& \G_0(JX + \sqrt{2}\ \sharp_g \xi + \sqrt{2}\ \flat_g X + \var J^*\xi, JY + \sqrt{2}\ \sharp_g \eta + \sqrt{2}\ \flat_g Y + \var J^*\eta)  \\
	  												=& \frac{1}{2}\left((\sqrt{2}\ \flat_g X + \var J^* \xi)(JY + \sqrt{2}\ \sharp_g \eta) + (\sqrt{2}\ \flat_g Y + \var J^*\eta)(JX + \sqrt{2}\ \sharp_g\xi)\right)  \\
	  												=& \frac{1}{2}\left(\sqrt{2}\ g(X, JY) + \var\xi(J^2 Y) + 2 g(X, \sharp_g \eta) + \sqrt{2}\ \var g(\sharp_g J^*\xi, \sharp_g \eta)\right.  \\
	  												&\left.+ \sqrt{2}\ g(Y, JX) + \var\eta(J^2 X) + 2 g(Y, \sharp_g \xi) + \sqrt{2}\ \var g(\sharp_g J^*\eta, \sharp_g \xi)\right)  \\
	  												=& \frac{\sqrt{2}(1 - \var)}{2}(g(X, JY) - g(\sharp_g J^*\xi, \sharp_g \eta)) + \frac{2 - \var}{2}(\xi(Y) + \eta(X))  \\
	  												=& \frac{\sqrt{2}(1 - \var)}{2}(g(X, JY) - g(\sharp_g J^*\xi, \sharp_g \eta)) + (2 - \var) \G_0(X + \xi, Y + \eta).
    \end{align*}
Therefore, $\F_{J, g}$ is compatible with $\G_0$ if and only if $\var = +1$. The computation for the twin metric associated to the structure $(\F_{J, g}, \G_0)$ when $\var = +1$ is the following:
	\begin{align*}
	  \Phi_{\F_{J, g}}(X + \xi, Y + \eta) &= \G_0(JX + \sqrt{2}\ \sharp_g \xi + \sqrt{2}\ \flat_g X + J^*\xi, Y + \eta) = \frac{1}{2}((\sqrt{2}\ \flat_g X + J^*\xi)(Y) + \eta(JX + \sqrt{2}\ \sharp_g\xi))  \\
	  									  &= \frac{1}{2}(\xi(JY) + \eta(JX)) + \frac{\sqrt{2}}{2}(g(X, Y) + g(\sharp_g\xi, \sharp_g\eta))  \\
	  									  &= \G_0(JX + \xi, JY + \eta) + \frac{\sqrt{2}}{2}(g(X, Y) + g(\sharp_g\xi, \sharp_g\eta)).
	\end{align*}
Hence, the result is proven.
\end{proof}

\begin{proposition}
Let $(M, F, g)$ be an $(\al, \var)$-metric manifold with $\al = +1$. Then, the generalized almost complex structure
	\begin{equation}
	  \J_{F, g} =
	  \left(
		\begin{array}{cc}
		  F        & -\sqrt{2}\ \sharp_g  \\
		  \sqrt{2}\ \flat_g  & -\var F^*
		\end{array}
	  \right),
	  \label{EQUATION:COMPLEXFG}
	\end{equation}
is compatible with $\G_0$ if and only if $\var = -1$, that is, if $(M, F, g)$ is an almost para-Hermitian manifold. In this case, $(\J_{F, g}, \G_0)$ is a natural generalized almost Norden structure. The twin metric associated to this structure is
	\begin{equation*}
	  \Phi_{\J_{F, g}}(X + \xi, Y + \eta) = \G_0(FX + \xi, FY + \eta) + \frac{\sqrt{2}}{2}(g(X, Y) - g(\sharp_g\xi, \sharp_g\eta)).
	\end{equation*}
\end{proposition}

\subsection{Induced generalized $(\al, \var)$-metric structures}

When the generalized metric considered on $\TM$ is not the canonical one, we can also obtain generalized $(\al, \var)$-metric structures. Here, we consider the generalized metric induced by a Riemannian or pseudo-Riemannian metric $g$ on the base manifold from Definition \ref{DEFINITION:INDUCEDGENERALIZEDMETRIC}.

\begin{definition}
Let $(M, g)$ be a Riemannian or pseudo-Riemannian manifold. A generalized $(\al, \var)$-metric structure is called \emph{induced} if it is an $(\al, \var)$-metric structure on $\TM$ with the generalized metric $\G_g$ from Eq. \eqref{EQUATION:INDUCEDGENERALIZEDMETRIC}.
\end{definition}

There are four viable kinds of induced generalized $(\al, \var)$-metric structures: generalized Riemannian or pseudo-Riemannian almost product structures (if $\G_g$ is Riemannian and the generalized polynomial structure is almost paracomplex, we call them generalized almost para-Norden structures), generalized almost para-Hermitian structures, generalized almost Hermitian or indefinite almost Hermitian structures, and generalized almost Norden structures.

We analyse first the induced generalized $(\al, \var)$-metric structures that can be obtained from a Riemannian or pseudo-Riemannian manifold $(M, g)$. As we have the generalized almost complex structure $\J_g$ from Eq. \eqref{EQUATION:COMPLEXG} and the generalized almost paracomplex structure $\F_g$ from Eq. \eqref{EQUATION:PRODUCTG}, the following proposition can be inferred.

\begin{proposition}
Let $(M, g)$ be a Riemannian or pseudo-Riemannian manifold. Then, the generalized almost complex structure $\J_g$ from Eq. \eqref{EQUATION:COMPLEXG} and the generalized almost paracomplex structure $\F_g$ from Eq. \eqref{EQUATION:PRODUCTG} are isometric with respect to the generalized metric $\G_g$. In other words, $(\J_g, \G_g)$ is a generalized almost Hermitian or indefinite almost Hermitian structure (depending on $g$) and $(\F_g, \G_g)$ is a generalized almost para-Norden or a pseudo-Riemannian almost product structure (depending on $g$). Their respective fundamental tensors are
	\begin{gather*}
	  \Phi_{\J_g}(X + \xi, Y + \eta) = -2\Omega_0(X + \xi, Y + \eta),  \\
	  \Phi_{\F_g}(X + \xi, Y + \eta) = 2\G_0(X + \xi, Y + \eta).
	\end{gather*}
\end{proposition}

\begin{proof}
We check that $\J_g, \F_g$ are isometric with respect to the generalized metric $\G_g$:
	\begin{gather*}
	  \G_g(\J_g(X + \xi), \J_g(Y + \eta)) = \G_g(-\sharp_g\xi + \flat_g X, -\sharp_g\eta + \flat_g Y) = g(-\sharp_g\xi, -\sharp_g\eta) + g(\sharp_g\flat_g X, \sharp_g\flat_g Y) = \G_g(X + \xi, Y + \eta),  \\
	  \G_g(\F_g(X + \xi), \F_g(Y + \eta)) = \G_g(\sharp_g\xi + \flat_g X, \sharp_g\eta + \flat_g Y) = g(\sharp_g\xi, \sharp_g\eta) + g(\sharp_g\flat_g X, \sharp_g\flat_g Y) = \G_g(X + \xi, Y + \eta).
	\end{gather*}
It is easy to check the expresions of the fundamental tensors:
	\begin{gather*}
	  \Phi_{\J_g}(X + \xi, Y + \eta) = \G_g(-\sharp_g\xi + \flat_g X, Y + \eta) = -g(\sharp_g\xi, Y) + g(\sharp_g\flat_g X, \sharp_g\eta) = -\xi(Y) + \eta(X) = -2\Omega_0(X + \xi, Y + \eta),  \\
	  \Phi_{\F_g}(X + \xi, Y + \eta) = \G_g(\sharp_g\xi + \flat_g X, Y + \eta) = g(\sharp_g\xi, Y) + g(\sharp_g\flat_g X, \sharp_g\eta) = \xi(Y) + \eta(X) = 2\G_0(X + \xi, Y + \eta).
	\end{gather*}
Therefore, the result is proven.
\end{proof}

This statement show an interesting connection between the canonical generalized structures and the induced generalized $(\al, \var)$-metric structures by any Riemannian or pseudo-Riemannian manifold $(M, g)$. This connection is given by the fact that $\G_g(\cdot, \cdot) = \G_0(2\F_g\cdot, \cdot)$ for any metric $g$, as it was indicated in Definition \ref{DEFINITION:INDUCEDGENERALIZEDMETRIC}.

We work now with any $(\al, \var)$-metric manifold $(M, J, g)$, with fundamental tensor $\phi(\cdot, \cdot) = g(J\cdot, \cdot)$. Then, we can induce the generalized structures $\J_{\phi}, \F_{\phi}$ from Eqs. (\ref{EQUATION:COMPLEXG}, \ref{EQUATION:PRODUCTG}) when $\al\var = +1$, and the generalized structures $\J_{\phi}, \F_{\phi}$ from Eqs. (\ref{EQUATION:COMPLEXOMEGA}, \ref{EQUATION:PRODUCTOMEGA}) when $\al\var = -1$. Their behaviour with respect to $\G_g$ is shown in the following statement.

\begin{proposition}
Let $(M, J, g)$ be an $(\al, \var)$-metric manifold with fundamental tensor $\phi(\cdot, \cdot) = g(J\cdot, \cdot)$. Then, the induced generalized structure $(\J_{\phi}, \G_g)$ is almost Hermitian or indefinite almost Hermitian (depending on $g$) when $\var = +1$, and almost Norden when $\var = -1$. On the other hand, $(\F_{\phi}, \G_g)$ is an induced generalized almost para-Norden or pseudo-Riemannian almost product structure (depending on $g$) for $\var = +1$, and a generalized almost para-Hermitian structure when $\var = -1$. When $\var = +1$, their respective fundamental tensors are
	\begin{gather*}
	  \Phi_{\J_{\phi}}(X + \xi, Y + \eta) = -2\Omega_0(JX + \xi, JY + \eta),  \\
	  \Phi_{\F_{\phi}}(X + \xi, Y + \eta) = 2\G_0(JX + \xi, JY + \eta),
	\end{gather*}
while for $\var = -1$ then
	\begin{gather*}
	  \Phi_{\J_{\phi}}(X + \xi, Y + \eta) = 2\G_0(JX + \xi, JY + \eta),  \\
	  \Phi_{\F_{\phi}}(X + \xi, Y + \eta) = -2\Omega_0(JX + \xi, JY + \eta).
	\end{gather*}
\end{proposition}

\begin{proof}
We check first that $\J_{\phi}$ and $\F_{\phi}$ are compatible with respect to $\G_g$. In order to do that, we use Proposition \ref{PROPOSITION:EQUALITIESFLATSHARP}:
	\begin{align*}
	  \G_g(\J_{\phi}(X + \xi), \J_{\phi}(Y + \eta)) &= \G_g(-\sharp_{\phi}\xi + \flat_{\phi}X, -\sharp_{\phi}\eta + \flat_{\phi}Y) = g(-\sharp_{\phi}\xi, -\sharp_{\phi}\eta) + g(\sharp_g\flat_{\phi} X, \sharp_g\flat_{\phi} Y)  \\
	  													&= g(\al J\sharp_g\xi, \al J\sharp_g\eta) + g(\sharp_g\flat_g J X, \sharp_g\flat_g J Y) = g(J\sharp_g\xi, J\sharp_g\eta) + g(JX, JY) \\
	  															  				&= \var (g(\sharp_g\xi, \sharp_g\eta) + g(X, Y))  \\
	  															  				&= \var \mathcal G_g(X + \xi, Y + \eta),  \\
	  \G_g(\F_{\phi}(X + \xi), \F_{\phi}(Y + \eta)) &= \G_g(\sharp_{\phi}\xi + \flat_{\phi}X, \sharp_{\phi}\eta + \flat_{\phi}Y) = g(\sharp_{\phi}\xi, \sharp_{\phi}\eta) + g(\sharp_g\flat_{\phi} X, \sharp_g\flat_{\phi} Y)  \\
	  																			&= \var \mathcal G_g(X + \xi, Y + \eta).
	\end{align*}
Hence, the generalized induced $(\al, \var)$-metric structures are the indicated ones. We compute now their respective fundamental tensors:
	\begin{align*}
	  \Phi_{\J_{\phi}}(X + \xi, Y + \eta) &= \G_g(-\sharp_{\phi}\xi + \flat_{\phi} X, Y + \eta) = g(-\sharp_{\phi}\xi, Y) + g(\sharp_g\flat_{\phi} X, \sharp_g\eta) = -\var g(\sharp_g J^*\xi, Y) + g(\sharp_g\flat_g JX, \sharp_g\eta)  \\
	                                        &= -\var\xi(JY) + \eta(JX),  \\
	  \Phi_{\F_{\phi}}(X + \xi, Y + \eta) &= \G_g(\sharp_{\phi}\xi + \flat_{\phi} X, Y + \eta) = g(\sharp_{\phi}\xi, Y) + g(\sharp_g\flat_{\phi} X, \sharp_g\eta) = \var g(\sharp_g J^*\xi, Y) + g(\sharp_g\flat_g JX, \sharp_g\eta)  \\
	                                        &= \var\xi(JY) + \eta(JX).
	\end{align*}
Therefore, the result is proven.
\end{proof}

We suppose now that $(M, J, g)$ is an $(\al, \var)$-metric manifold with $\al = -1$ (that is, $J^2 = -Id$). As $J$ is an almost complex structure, we can induce the generalized almost complex structures $\J_{\lambda, J}$ from Eq. \eqref{EQUATION:COMPLEXLAMBDAJ} with $\lambda\in \{+1, -1\}$, the generalized almost complex structures $\J_{J, g, \flat}, \J_{J, g, \sharp}$ from Eq. \eqref{EQUATION:TRIANGULARMJG}, and the generalized almost paracomplex structure $\F_{J, g}$ from Eq. \eqref{EQUATION:PRODUCTJG}. Although it is easy to check that $\J_{J, g, \flat}, \J_{J, g, \sharp}$ are not compatible with the metric $\G_g$ (in fact, they are compatible with the generalized metric defined in Eq. \eqref{EQUATION:NANNICINIEXAMPLEMETRIC}, see \cite[Sec. 3]{NANNICINI2019} for the case $\var = -1$), the other structures are compatible with $\G_g$. The following propositions gather these results.

\begin{proposition}
\label{PROPOSITION:ALVARINDUCEDCOMPLEXLAMBDAJ}
Let $(M, J, g)$ be an $(\al, \var)$-metric manifold with $\al = -1$ and $\phi(\cdot, \cdot) = g(J\cdot, \cdot)$ its fundamental tensor. Then, the generalized almost complex structure $\J_{\lambda, J}$ from Eq. \eqref{EQUATION:COMPLEXLAMBDAJ}, with $\lambda\in \{+1, -1\}$, is compatible with the induced generalized metric $\G_g$. In particular, when $\var = +1$ the induced structure $(\J_{\lambda, J}, \G_g)$ is a generalized almost Hermitian or indefinite almost Hermitian structure (depending on $g$); and when $\var = -1$ the induced structure $(\J_{\lambda, J}, \G_g)$ is a generalized almost Norden structure. The fundamental tensor of $\ (\J_{\lambda, J}, \G_g)$ is
	\begin{equation*}
	  \Phi_{\J_{\lambda, J}}(X + \xi, Y + \eta) = \phi(X, Y) - \lambda\phi(\sharp_{\phi}\xi, \sharp_{\phi}\eta).
	\end{equation*}
\end{proposition}

\begin{proof}
First, we check the compatibility between the generalized almost complex structure $\J_{\lambda, J}$ and the generalized induced metric $\G_g$:
	\begin{align*}
	  \G_g(\J_{\lambda, J}(X + \xi), \J_{\lambda, J}(Y + \eta)) &= \G_g(JX + \lambda J^*\xi, JY + \lambda J^*\eta) = g(JX, JY) + g(\lambda\sharp_g J^*\xi, \lambda\sharp_g J^*\eta)  \\
	  													&= g(JX, JY) + g(J\sharp_g\xi, J\sharp_g\eta) = \var(g(X, Y) + g(\sharp_g\xi, \sharp_g\eta)) \\
	  															  				&= \var \G_g(X + \xi, Y + \eta).
	\end{align*}
The fundamental tensor of the structure can be easily obtained:
	\begin{align*}
	  \Phi_{\J_{\lambda, J}}(X + \xi, Y + \eta) &= \G_g(JX + \lambda J^*\xi, Y + \eta) = g(JX, Y) + \lambda g(\sharp_g J^*\xi, \sharp_g\eta) = \phi(X, Y) + \lambda\var g(\sharp_{\phi}\xi, J\sharp_{\phi}\eta)  \\
	  											&= \phi(X, Y) - \lambda g(J\sharp_{\phi}\xi, \sharp_{\phi}\eta)  \\
	  										    &= \phi(X, Y) - \lambda\phi(\sharp_{\phi}\xi, \sharp_{\phi}\eta).
	\end{align*}
Hence, we finish the proof.
\end{proof}

\begin{proposition}
\label{PROPOSITION:ALVARINDUCEDPARACOMPLEXJG}
Let $(M, J, g)$ be an $(\al, \var)$-metric manifold with $\al = -1$. Then, the generalized almost paracomplex structure $\F_{J, g}$ from Eq. \eqref{EQUATION:PRODUCTJG} is compatible with the induced generalized metric $\G_g$ if and only if $\var = -1$, that is, if $(M, J, g)$ is an almost Norden manifold. In this case, $(\F_{J, g}, \G_g)$ is an induced generalized pseudo-Riemannian almost product structure. If we denote the twin metric of the $(\al, \var)$-metric manifold $(M, J, g)$ as $\phi$, the twin metric of the induced structure $(\F_{J, g}, \G_g)$ is given by
	\begin{equation*}
	  \Phi_{\F_{J, g}}(X + \xi, Y + \eta) = 2\sqrt{2}\ \G_0(X + \xi, Y + \eta) + \phi(X, Y) + \phi(\sharp_{\phi}\xi, \sharp_{\phi}\eta).
	\end{equation*}
\end{proposition}

\begin{proof}
We show here the calculations to check the compatibility between the generalized almost paracomplex structure $\F_{J, g}$ and the induced generalized metric $\G_g$, taking into account Proposition \ref{PROPOSITION:EQUALITIESFLATSHARP}:
	\begin{align*}
	  \G_g(\F_{J,g}(X + \xi), \F_{J,g}(Y + \eta)) =& \G_g(JX + \sqrt{2}\ \sharp_g\xi + \sqrt{2}\ \flat_g X + \var J^*\xi, JY + \sqrt{2}\ \sharp_g\eta + \sqrt{2}\ \flat_g Y + \var J^*\eta)  \\
	  											  =& g(JX + \sqrt{2}\ \sharp_g\xi, JY + \sqrt{2}\ \sharp_g\eta) + g(\sharp_g(\sqrt{2}\ \flat_g X + \var J^*\xi), \sharp_g(\sqrt{2}\ \flat_g Y + \var J^*\eta))  \\
	  											  =& g(JX, JY) + \sqrt{2}\ g(\sharp_g\xi, JY) + \sqrt{2}\ g(JX, \sharp_g\eta) + 2g(\sharp_g\xi, \sharp_g\eta)  \\
	  											   &+ 2g(X, Y) + \var\sqrt{2}\ g(\sharp_g J^*\xi, Y) + \var\sqrt{2}\ g(X, \sharp_g J^*\eta) + g(\sharp_g J^*\xi, \sharp_g J^*\eta)  \\
	  											  =& \sqrt{2}(1 + \var)(\xi(JY) + \eta(JX)) + (2 + \var)(g(X, Y) + g(\sharp_g\xi, \sharp_g\eta))  \\
	  											  =& \sqrt{2}(1 + \var)(\xi(JY) + \eta(JX)) + (2 + \var)\G_g(X + \xi, Y + \eta).
	\end{align*}
Therefore, the generalized almost paracomplex structure $\F_{J, g}$ is compatible with $\G_g$ if and only if $\var = -1$, that is, if $(M, J, g)$ is an almost Norden manifold. As $\G_g$ must be a generalized pseudo-Riemannian metric with neutral signature, $(\F_{J, g}, \G_g)$ is an induced generalized pseudo-Riemannian almost product structure. We compute its twin metric:
	\begin{align*}
	  \Phi_{\F_{J, g}}(X + \xi, Y + \eta) &= \G_g(JX + \sqrt{2}\ \sharp_g\xi + \sqrt{2}\ \flat_g X - J^*\xi, Y + \eta) = g(JX + \sqrt{2}\ \sharp_g\xi, Y) + g(\sharp_g(\sqrt{2}\ \flat_g X - J^*\xi), \sharp_g\eta)  \\
	  									  &= g(JX, Y) + \sqrt{2}\ g(\sharp_g\xi, Y) + \sqrt{2}\ g(X, \sharp_g\eta) - g(\sharp_g J^*\xi, \sharp_g\eta)  \\
	  									  &= \sqrt{2}(\xi(Y) + \eta(X)) + \phi(X, Y) + \phi(\sharp_{\phi}\xi, \sharp_{\phi}\eta)  \\
	  								      &= 2\sqrt{2}\ \G_0(X + \xi, Y + \eta) + \phi(X, Y) + \phi(\sharp_{\phi}\xi, \sharp_{\phi}\eta).
	\end{align*}
Therefore, the result is proven.
\end{proof}

Analogously, when $(M, F, g)$ is an $(\al, \var)$-metric manifold with $\al = +1$ (that is, $F^2 = +Id$), we can induce the generalized almost product structures $\F_{\lambda, F}$ from Eq. \eqref{EQUATION:PRODUCTLAMBDAF} with $\lambda\in \{+1, -1\}$, the generalized almost product structures $\F_{F, g, \flat}, \F_{F, g, \sharp}$ from Eq. \eqref{EQUATION:TRIANGULARMFG}, and the generalized almost complex structure $\J_{F, g}$ from Eq. \eqref{EQUATION:COMPLEXFG}. The following propositions study these structures. As their proofs are analogous to the proofs of Propositions \ref{PROPOSITION:ALVARINDUCEDCOMPLEXLAMBDAJ}, \ref{PROPOSITION:ALVARINDUCEDPARACOMPLEXJG}, they are not specified here.

\begin{proposition}
Let $(M, F, g)$ be an $(\al, \var)$-metric manifold with $\al = +1$ and $\phi(\cdot, \cdot) = g(F\cdot, \cdot)$ its fundamental tensor. Then, the generalized almost product structure $\F_{\lambda, F}$ from Eq. \eqref{EQUATION:PRODUCTLAMBDAF}, with $\lambda\in \{+1, -1\}$, is compatible with the induced generalized metric $\G_g$. In particular, when $\var = +1$ the induced structure $(\F_{\lambda, F}, \G_g)$ is a generalized Riemannian or pseudo-Riemannian almost product structure (depending on $g$); and when $\var = -1$ the induced structure $(\F_{\lambda, F}, \G_g)$ is a generalized almost para-Hermitian structure. The fundamental tensor of $(\F_{\lambda, F}, \G_g)$ is
	\begin{equation*}
	  \Phi_{\F_{\lambda, F}}(X + \xi, Y + \eta) = \phi(X, Y) + \lambda\phi(\sharp_{\phi}\xi, \sharp_{\phi}\eta).
	\end{equation*}
\end{proposition}

\begin{proposition}
Let $(M, F, g)$ be an $(\al, \var)$-metric manifold with $\al = +1$. Then, the generalized almost complex structure $\J_{F, g}$ from Eq. \eqref{EQUATION:COMPLEXFG} is compatible with the induced generalized metric $\G_g$ if and only if $\var = -1$, that is, if $(M, F, g)$ is an almost para-Hermitian manifold. In this case, $(\J_{F, g}, \G_g)$ is an induced generalized indefinite almost Hermitian structure. If we denote the fundamental symplectic structure of the $(\al, \var)$-metric manifold $(M, F, g)$ as $\phi$, the fundamental symplectic structure of $(\J_{F, g}, \G_g)$ is given by
	\begin{equation*}
	  \Phi_{\J_{F, g}}(X + \xi, Y + \eta) = -2\sqrt{2}\ \Omega_0(X + \xi, Y + \eta) + \phi(X, Y) + \phi(\sharp_{\phi}\xi, \sharp_{\phi}\eta).
	\end{equation*}
\end{proposition}

\section{Generalized triple structures}
\label{SECTION:GENERALIZEDTRIPLESTRUCTURES}

In this section, we want to check the interactions between the multiple generalized polynomial structures that we have presented before. Different relations of commutation or anti-commutation can be found in the examples presented in Section \ref{SECTION:GENERALIZEDALVARMETRICSTRUCTURES}, based principally on Proposition \ref{PROPOSITION:EQUALITIESFLATSHARP}. These generalized polynomial structures, following the line given in Definition \ref{DEFINITION:TRIPLESTRUCTUREVECTORBUNDLE}, conform triple structures on the generalized tangent bundle.

\begin{definition}
A triple structure $(\F, \J, \K)$ on the generalized tangent bundle $\TM$ is called a \emph{generalized triple structure}.
\end{definition}

There are four possible kinds of generalized triple structures, namely, generalized almost hypercomplex structures, generalized almost bicomplex structures, generalized almost biparacomplex structures and generalized almost hyperproduct structures.

Many of these structures appear because of the existence of the canonical generalized almost paracomplex structure $\F_0$. In the following proposition, that is straightforward to prove, we show the conditions that a generalized polynomial structure must fulfil in order to commute or anti-commute with $\F_0$.

\begin{proposition}
\label{PROPOSITION:COMMUTATIONANTICOMMUTATIONF0}
Let $\K\colon \GTM\to \GTM$ be a generalized polynomial structure such that it can be written in matrix notation as in Eq. \eqref{EQUATION:MATRIXNOTATION}, and let $\F_0$ be the natural generalized almost paracomplex structure. Then, $\K$ and $\F_0$ commute if and only if $\tau = 0$ and $\sigma = 0$. On the other hand, $\K$ and $\F_0$ anti-commute if and only if $H = 0$ and $K = 0$.
\end{proposition}

Proposition \ref{PROPOSITION:COMMUTATIONANTICOMMUTATIONF0} can be used in order to find generalized triple structures that involve the generalized polynomial structures shown in this document.

\begin{corollary}
Let $(M, g)$ be a Riemannian or pseudo-Riemannian manifold. Then, the structures $\J_g$ and $\F_g$ from Eqs. \emph{(\ref{EQUATION:COMPLEXG}, \ref{EQUATION:PRODUCTG})} anti-commute with the natural generalized almost paracomplex structure $\F_0$. Specifically, we have that $\J_g = \F_0 \F_g$, so $(\F_0, \F_g, \J_g)$ is a generalized almost biparacomplex structure.
\end{corollary}

\begin{corollary}
Let $(M, \omega)$ be an almost symplectic manifold. Then, $\J_{\omega}$ and $\F_{\omega}$  from Eqs. \emph{(\ref{EQUATION:COMPLEXOMEGA}, \ref{EQUATION:PRODUCTOMEGA})} anticommute with the natural generalized almost paracomplex structure $\F_0$. Specifically, we have that $\J_{\omega} = \F_0 \F_{\omega}$, so $(\F_0, \F_{\omega}, \J_{\omega})$ is a generalized almost biparacomplex structure.
\end{corollary}

\begin{corollary}
Let $(M, J)$ be an almost complex manifold. Then, the generalized almost complex structures $\J_{\lambda, J}$ for $\lambda\in \{+1, -1\}$ from Eq. \eqref{EQUATION:COMPLEXLAMBDAJ} commute with the natural generalized almost paracomplex structure $\F_0$. Specifically, we have that $\F_0 = \J_{+1, J} \J_{-1, J}$, so $(\J_{+1, J}, \J_{-1, J}, \F_0)$ is a generalized almost bicomplex structure.
\end{corollary}

\begin{corollary}
Let $(M, F)$ be an almost product manifold. Then, the generalized almost product structures $\F_{\lambda, F}$ for $\lambda\in \{+1, -1\}$ from Eq. \eqref{EQUATION:PRODUCTLAMBDAF} commute with the natural generalized almost paracomplex structure $\F_0$. Specifically, we have that $-\F_0 = \F_{+1, F} \F_{-1, F}$, so $(\F_{+1, F}, \F_{-1, F}, -\F_0)$ is a generalized almost hyperproduct structure.
\end{corollary}

We consider now a richer geometry within the base manifold. First, we take an $(\al, \var)$-metric manifold $(M, J, g)$ with $\al = -1$. We can work with many induced generalized polynomial structures: the ones induced by $J$, by the metric $g$ and by its fundamental tensor $\phi(\cdot, \cdot) = g(J\cdot, \cdot)$. We check first how the generalized polynomial structures $\J_{\lambda, J}$ interact with $\J_g$, keeping in mind Proposition \ref{PROPOSITION:EQUALITIESFLATSHARP}:
	\begin{gather*}
	  \J_{\lambda, J} \J_g = \left(
        \begin{array}{cc}
          J  & 0  \\
          0  & \lambda J^*
        \end{array}
      \right) \left(
        \begin{array}{cc}
          0        & -\sharp_g  \\
          \flat_g  & 0
        \end{array}
      \right) = \left(
        \begin{array}{cc}
          0        & -J\sharp_g  \\
          \lambda J^*\flat_g  & 0
        \end{array}
      \right) = \left(
        \begin{array}{cc}
          0        & \sharp_{\phi}  \\
          -\lambda\var\flat_{\phi}  & 0
        \end{array}
      \right),  \\
	  \J_g \J_{\lambda, J} = \left(
        \begin{array}{cc}
          0        & -\sharp_g  \\
          \flat_g  & 0
        \end{array}
      \right) \left(
        \begin{array}{cc}
          J  & 0  \\
          0  & \lambda J^*
        \end{array}
      \right) = \left(
        \begin{array}{cc}
          0        & -\lambda\sharp_g J^*  \\
          \flat_g J  & 0
        \end{array}
      \right) = \left(
        \begin{array}{cc}
          0        & -\lambda\var\sharp_{\phi}  \\
          \flat_{\phi}  & 0
        \end{array}
      \right).	
	\end{gather*}
We do the same for $\J_{\lambda, J}$ and $\F_g$:
	\begin{gather*}
	  \J_{\lambda, J} \F_g = \left(
        \begin{array}{cc}
          J  & 0  \\
          0  & \lambda J^*
        \end{array}
      \right) \left(
        \begin{array}{cc}
          0        & \sharp_g  \\
          \flat_g  & 0
        \end{array}
      \right) = \left(
        \begin{array}{cc}
          0        & J\sharp_g  \\
          \lambda J^*\flat_g  & 0
        \end{array}
      \right) = \left(
        \begin{array}{cc}
          0        & -\sharp_{\phi}  \\
          -\lambda\var\flat_{\phi}  & 0
        \end{array}
      \right),  \\
	  \F_g \J_{\lambda, J} = \left(
        \begin{array}{cc}
          0        & \sharp_g  \\
          \flat_g  & 0
        \end{array}
      \right) \left(
        \begin{array}{cc}
          J  & 0  \\
          0  & \lambda J^*
        \end{array}
      \right) = \left(
        \begin{array}{cc}
          0        & \lambda\sharp_g J^*  \\
          \flat_g J  & 0
        \end{array}
      \right) = \left(
        \begin{array}{cc}
          0        & \lambda\var\sharp_{\phi}  \\
          \flat_{\phi}  & 0
        \end{array}
      \right).	
	\end{gather*}
Then, we state the following result that combine all these calculations.
\begin{proposition}
\label{PROPOSITION:TRIPLEMJG}
Let $(M, J, g)$ be an $(\al, \var)$-metric manifold with $\al = -1$, and let $\phi(\cdot, \cdot) = g(J\cdot, \cdot)$ be its fundamental tensor. Then, we find the following generalized triple structures: a generalized almost hypercomplex structure $(\J_g, \J_{\var, J}, \J_{\phi})$; a generalized almost bicomplex structure $(\J_g, \J_{-\var, J}, \F_{\phi})$; other generalized almost bicomplex structure $(\J_{\phi}, \J_{-\var, J}, -\F_g)$; and a generalized almost biparacomplex structure $(\F_g, \F_{\phi}, \J_{\var, J})$.
\end{proposition}

One interesting thing to comment is the relation between the generalized almost biparacomplex structure $(\F_g, \F_{\phi}, \J_{\var, J})$ and the induced structure $\F_{J, g}$ from Eq. \eqref{EQUATION:PRODUCTJG}: as all the structures in $(\F_g, \F_{\phi}, \J_{\var, J})$ anti-commute, every combination in the form $a\F_g + b\F_{\phi} + c\J_{\var, J}$ for $a, b, c\in \fm$ such that $-a^2 - b^2 + c^2 = 1$ is a generalized almost complex structure, whilst if  $a^2 + b^2 - c^2 = 1$ then the combination is a generalized almost product structure. It can be easily seen that $\F_{J, g}$ is the combination of $\F_g, \F_{\phi}, \J_{\var, J}$ for the values $a = \sqrt{2}, b = 0, c = 1$, that is, $\F_{J, g} = \sqrt{2}\F_g + \J_{\var, J}$.

Analogous calculations can be done for an $(\al, \var)$-metric manifold $(M, F, g)$ with $\al = +1$ (that is, $F$ is almost product). We resume those calculations in the following proposition.

\begin{proposition}
Let $(M, F, g)$ be an $(\al, \var)$-metric manifold with $\al = +1$, and let $\phi(\cdot, \cdot) = g(F\cdot, \cdot)$ be its fundamental tensor. Then, we can induce the following generalized triple structures: a generalized almost hyperproduct structure $(\F_g, \F_{\var, F}, \F_{\phi})$; a generalized almost biparacomplex structure $(\F_g, \F_{-\var, F}, \J_{\phi})$; one more generalized almost biparacomplex structure $(\F_{\phi}, \F_{-\var, F}, \J_g)$; and a generalized almost bicomplex structure $(\J_g, \J_{\phi}, -\F_{\var, F})$.
\end{proposition}

As for the case of the structure $\F_{J, g}$, the behaviour of the generalized almost complex morphism $\J_{F, g}$ from Eq. \eqref{EQUATION:COMPLEXFG} is determined by the generalized almost biparacomplex structure $(\F_{\phi}, \F_{-\var, F}, \J_g)$. This three morphisms anti-commute, so a combination like $a\F_{\phi} + b\F_{-\var, F} + c\J_g$ for $a, b, c\in \fm$ behaves as a generalized almost product structure if $a^2 + b^2 - c^2 = 1$, and as an almost complex structure if $-a^2 - b^2 + c^2 = 1$. Then, the morphism $\J_{F, g}$ can be expressed as $\J_{F, g} = \sqrt{2}\J_g + \F_{-\var, F}$.

To end, we comment a specific case of generalized almost bicomplex structures that have been widely studied (see, for example, \cite{VANDERLEERDURAN2018, LINTOLMAN2006}). These structures are called generalized Kähler structures.

\begin{definition}[{\cite[Def. 1.8]{GUALTIERI2014}}]
\label{DEFINITION:GENERALIZEDKAHLER}
A \emph{generalized Kähler structure} is defined as an almost bicomplex structure $(\J_1, \J_2, \F)$ such that $\J_1$ and $\J_2$ are integrable with respect to the Courant bracket (see \cite[Sec. 2]{HITCHIN2003}), isometric with respect to the canonical metric $\G_0$ (that is, they form natural generalized indefinite almost Hermitian structures) and $-\F$ induces a generalized Riemannian metric as it is shown in Eq. \eqref{EQUATION:CHARACTERIZATIONGENERALIZEDMETRICS}.
\end{definition}

As it was commented in Section \ref{SECTION:INTRODUCTION}, in this paper we do not focus in the integrability of polynomial structures. Therefore, if we omit this condition from Definition \ref{DEFINITION:GENERALIZEDKAHLER} we obtain \emph{generalized almost Kähler structures}. The integrability conditions of these structures are detailed in \cite[Sec. 1.3]{GUALTIERI2014}.

The main example of a generalized almost Kähler structure is induced by an almost Hermitian manifold $(M, J, g)$ with $\phi(\cdot, \cdot) = g(J\cdot, \cdot)$ as its fundamental symplectic structure. In this case, the generalized almost bicomplex structure $(\J_{\phi}, \J_{-1, J}, -\F_g)$ presented in Proposition \ref{PROPOSITION:TRIPLEMJG} conforms a generalized almost Kähler structure.

Of course, some of the conditions asked for a generalized almost Kähler structure could be relaxed in order to study a wider range of generalized structures. For example, if we omitted the isometry condition with respect to $\G_0$, an almost Norden manifold $(M, J, g)$ with twin metric $\phi$ would induce a generalized Kähler structure from the generalized almost bicomplex structure $(\J_{\phi}, \J_{+1, J}, -\F_g)$.

There are generalized almost bicomplex structures that do not behave as a generalized Kähler structure. One example is the generalized structure $(\J_{+1, J}, \J_{-1, J}, \F_0)$: the canonical generalized almost paracomplex structure $\F_0$ does not define a Riemannian metric.

To end, we just recall a result from \cite{GUALTIERI2014} that relates each generalized almost Kähler structure on $\TM$ with a structures that is defined on the manifold $M$.

\begin{theorem}[{\cite[Thm. 1.12]{GUALTIERI2014}}]
Each generalized almost Kähler structure $(\J_1, \J_2, \F)$ on $\TM$ is equivalent to a structure $(b, g, J_1, J_2)$ determined by a 2-form $b\in \Lambda^2(M)$ on $M$, a Riemannian metric $g$ on $M$ and two almost complex structures $J_1, J_2$ that are isometric with respect to $g$. The generalized almost complex structures $\J_1, \J_2$ can be recovered from the structure $(b, g, J_1, J_2)$ using the expression
	\begin{equation*}
	  \J_{1/2} = \frac{1}{2}
	  \left(
	  \begin{array}{cc}
	    Id & 0 \\
	    \sigma  & Id
	  \end{array}
	  \right)
	  \left(
	  \begin{array}{cc}
	    J_1\pm J_2        				  & -(\sharp_{\phi_1}\mp \sharp_{\phi_2}) \\
	    \flat_{\phi_1}\mp \flat_{\phi_2}  & -(J_1^*\pm J_2^*)
	  \end{array}
	  \right)
	  \left(
	  \begin{array}{cc}
	    Id 		 & 0 \\
	    -\sigma  & Id
	  \end{array}
	  \right),
	\end{equation*}
where $\sigma\colon \xm\to \dm$ is defined as $(\sigma X)(Y) = b(X, Y)$, and $\phi_i$ is given by $\phi_i(\cdot, \cdot) = g(J_i\cdot, \cdot)$ for $i = 1, 2$.
\end{theorem}


\end{document}